\documentclass[journal,twoside,web]{ieeecolor}
\usepackage{generic}

\usepackage{cite}
\usepackage{amsmath,amssymb,amsfonts}
\usepackage{graphicx}
\usepackage{textcomp}

\usepackage{algcompatible}
\usepackage{algorithm}

\usepackage{bm}

\newtheorem{theorem}{Theorem}
\newtheorem{lemma}{Lemma}
\newtheorem{definition}{Definition}
\newtheorem{corollary}{Corollary}

\newtheorem{remark}{Remark}

\DeclareMathOperator*{\argmin}{arg\,min}

\DeclareMathOperator{\stat}{stat}

\newcommand{\Tr}{^{\top}}
\newcommand{\reg}{\texttt{Regret}}
\renewcommand{\Re}{\mathbb{R}}
\newcommand{\NN}{\mathbb{N}}

\newcommand{\dx}{\mathrm{d}}
\newcommand{\xx}{\mathbf{x}}
\newcommand{\xt}{\mathbf{x}_t}
\newcommand{\xtt}{\mathbf{x}_{t+1}}
\newcommand{\xtm}{\mathbf{x}_{t-1}}
\newcommand{\xs}{\xt^\ast}

\newcommand{\vv}{\mathbf{v}}

\renewcommand{\aa}{\mathbf{a}}
\renewcommand{\ss}{\mathbf{s}}

\newcommand{\MM}{\mathbf{M}}
\newcommand{\HH}{\mathbf{H}}

\newcommand{\hes}{\nabla^2 f_t(\xx_t)}

\newcommand{\XX}{\mathcal{X}}

\newcommand{\ox}{\overline{x}}
\newcommand{\ux}{\underline{x}}

\newcommand{\ov}{\overline{v}}

\newcommand{\oV}{\overline{V}}
\newcommand{\oE}{\overline{E}}
\newcommand{\uE}{\underline{E}}

\def\BibTeX{{\rm B\kern-.05em{\sc i\kern-.025em b}\kern-.08em
    T\kern-.1667em\lower.7ex\hbox{E}\kern-.125emX}}

\begin{document}
\title{Second-order Online Nonconvex Optimization}

\author{Antoine Lesage-Landry,~\IEEEmembership{Member, IEEE}, Joshua A. Taylor,~\IEEEmembership{Member, IEEE}, and Iman~Shames,~\IEEEmembership{Member, IEEE}
\thanks{This work was funded by the Fonds de recherche du Qu\'ebec -- Nature et technologies, the Ontario Ministry of Research, Innovation and Science and the Natural Sciences and Engineering Research Council of Canada, Defence Science and Technology Group, through agreement MyIP: ID9156 entitled ``Verifiable Hierarchical Sensing, Planning and Control'', the Australian Government, via grant AUSMURIB000001 associated with ONR MURI grant N00014-19-1-2571}
\thanks{A. Lesage-Landry is with the Energy \& Resources Group, University of California, Berkeley, CA, USA (e-mail: \texttt{alesagelandry@berkeley.edu}).}
\thanks{J.A. Taylor is with The Edward S. Rogers Sr. Department of Electrical \& Computer Engineering, University of Toronto, Ontario, Canada (e-mail: \texttt{josh.taylor@utoronto.ca}).}
\thanks{I. Shames is with the Department of Electrical and Electronic Engineering, University of Melbourne, Australia (e-mail:  \texttt{ishames@unimelb.edu.au}).}
\thanks{This work was done while A. Lesage-Landry was visiting the University of Melbourne, Australia.}
}

\maketitle

\begin{abstract}
We present the online Newton's method, a single-step second-order method for online nonconvex optimization. We analyze its performance and obtain a dynamic regret bound that is linear in the cumulative variation between round optima. We show that if the variation between round optima is limited, the method leads to a constant regret bound. In the general case, the online Newton's method outperforms online convex optimization algorithms for convex functions and performs similarly to a specialized algorithm for strongly convex functions. We simulate the performance of the online Newton's method on a nonlinear, nonconvex moving target localization example and find that it outperforms a first-order approach.
\end{abstract}


\section{Introduction}
\label{sec:introduction}
\IEEEPARstart{I}{n} online or time-varying optimization one must sequentially provide decisions based only on past information. This problem arises in modeling resource allocation problems in networks~\cite{chen2017online,cao2018virtual}, real-time deployment in electric power systems~\cite{ma2016distributed,lesage2018setpoint}, and localization of moving targets problems as in~\cite{bedi2018tracking,dixit2019online}. 

We consider online optimization problems of the following form. Let $\xt \in \Re^n$ be the decision vector at time $t$. Let $f_t:\Re^n \mapsto \Re$ be a twice differentiable function. We do not require it to be convex. The problems are of the form
\begin{equation}
\min_{\xx_t} f_t\left( \xx_t \right) \label{eq:opt_pro}
\end{equation}
for $t = 1, 2, \ldots, T$ where $T$ is the time horizon. The decision maker must solve~\eqref{eq:opt_pro} at each round $t$ using the information from rounds $t-1, t-2, \ldots, 0$. The objective function $f_t$ is observed when the round $t$ ends. The goal is to provide real-time decisions when information, time and/or resources are too limited to solve~\eqref{eq:opt_pro}. We base our analysis on online convex optimization~\cite{zinkevich2003online,shalev2012online,hazan2016introduction} (OCO). We characterize our approach in terms of the dynamic regret, defined as
\[
\reg(T) = \sum_{t=0}^T f_t(\xt) - f_t(\xs),
\]
where $\xs \in \argmin_{\xx \in \XX} f_t(\xx)$. In the case of static regret, $\xs$ is replaced by $\xx^\ast \in \argmin_{\xx \in \XX} \sum_{t=1}^T f_t(\xx)$. The static regret does not capture changes in the optimal solution, and for this reason we only work with dynamic regret. Our objective is to design algorithms with a sublinear regret in the number of rounds. A sublinear regret implies that, on average, the algorithm plays the optimal decision~\cite{hazan2016introduction,mokhtari2016online}. In this work, we restrict ourselves to unconstrained problems. Constrained online optimization is a topic for future work.

We propose the online Newton's method (\texttt{ONM}) and show that its dynamic regret is $O\left(V_T + 1 \right)$ where $V_T$ is the cumulative variation between round optima. \texttt{ONM} is an online nonconvex optimization algorithm, which only assumes local Lipschitz properties. 
We acknowledge that, to date, no regret analysis has been given for first-order online approaches, e.g., online gradient descent~\cite{zinkevich2003online} (\texttt{OGD}), under the current assumptions. However, given their convexity requirement and the poor performance of \texttt{OGD} on the example in Section~\ref{sec:example}, we believe that first-order approaches are unlikely to have as good a bound as \texttt{ONM}.
We obtain a bound on the regret of \texttt{ONM} of the same order as OCO methods for strongly convex functions when the initial point is in a neighborhood of the global optimum, and the variation between stationary optima is bounded. We also provide a constant regret bound for settings where the total variation between round optima is small. Moreover, \texttt{OMN} can be used to solve problems of the form $\stat_{\xx} f_t\left( \xx \right)$, i.e., to track stationary points of $f_t$ under the aforementioned assumptions. In this case, $\xs \in \left\{ \left.\xx \in \Re^n \right| \nabla f_t \left(\xx \right) = \mathbf{0}\right\}$ in the dynamic regret definition, and the same regret analysis holds.

We present a numerical example in which \texttt{OMN} is used to track a moving target from noisy measurements (see~\cite{cao2006sensor} and references therein). The online moving target localization problem is nonconvex and thus conventional OCO algorithms have no guarantee on their performance. We test the performance of \texttt{ONM} on a moving target localization example and find that it outperforms a gradient-based OCO algorithm.

\emph{Related work.} To the best of the our knowledge, \texttt{ONM} is the first dynamic regret-bounded, single-step, second-order online approach. \texttt{ONM} applies to general smooth nonconvex functions. An online damped Newton method is proposed in~\cite{zhang2017improved} but requires the objective function to be strongly convex and self-concordant, and multiple Newton steps must be performed at each round to obtain a sublinear dynamic regret bound. Reference~\cite{gao2018online} developed gradient-based algorithms for weakly pseudo-convex functions. An online approach using multiple gradient step-like updates at each round is proposed in~\cite{hazan2017efficient} to minimize the local regret of nonconvex loss functions. Reference~\cite{hazan2007logarithmic} proposed the first Newton step-like approach, in which the Hessian is approximated by the outer product of the  gradient, but only provided a static regret analysis. This work provides the tightest regret bound to date at $O\left(\log T\right)$, but, to the best of our knowledge, no analysis based on dynamic regret has been published. Several authors have proposed dynamic regret bounded approaches. These approaches include, for example, \texttt{OGD}~\cite{zinkevich2003online}, in which the author proposed the first dynamic bound in terms of $V_T$, a specialized version for $\sigma$-strongly convex functions ($\sigma\texttt{OGD}$)~\cite{mokhtari2016online}, dynamic mirror descent (\texttt{DMD})~\cite{hall2015online}, a specialized optimistic mirror descent for predictable sequences (\texttt{OMD})~\cite{rakhlin2013online}, and several other context-specific algorithms, e.g.,~\cite{besbes2015non,jadbabaie2015online,yi2016tracking,shahrampour2017distributed,lesage2018predictive,li2018online,zhao2018proximal,bedi2018tracking,dixit2019online,bernstein2019online}. 

Related work in parametric optimization has also investigated Newton step-like methods for time-varying nonconvex optimization~\cite{zavala2010real,dontchev2013euler}. Reference~\cite{simonetto2016class} presents a discrete time-sampling method consisting of a prediction and a Newton-based correction step for continuous-time time-varying convex optimization. Reference~\cite{tang2018feedback} proposed a regularized primal-dual algorithm to track solutions of time-varying nonconvex optimization problems. A time-varying Quasi-Newton method was investigated in the context of the nonconvex optimal power flow in~\cite{tang2017real}. Their approach, however, requires the approximate solution of a quadratic program at each round instead of a single update rule as in the present work. Our work also differs from the aforementioned references in that our algorithm can track any type of stationary point, and we characterize its performance in terms of dynamic regret.

\section{Preliminaries}
In this section, we introduce the second-order update and several technical results we will use to prove the main theorems.

\subsection{Background}

Let $\HH_t\left(\xt\right) = \hes$, the Hessian matrix at round $t$. The online Newton's method update is a Newton step with a unit step size. Throughout this work, we assume that the update is defined for all $t=1,2,\ldots,T$.
\begin{definition}[Online Newton update]
The online Newton update is:
\[
\xtt = \xt - \HH_{t}^{-1} \left(\xt \right) \nabla f_t \left( \xt \right).
\]
\label{def:newton_update}
\end{definition}

Let $\vv_t \in \Re^n$, the variation in optimum, be defined as:
\begin{align}
\xtt^\ast &= \xs + \vv_t. \label{def:v}
\end{align}
The total variation is
\[
V_T = \sum_{t=1}^T \left\| \xtm^\ast - \xs \right\| = \sum_{t=0}^{T-1} \left\| \vv_t \right\|.
\]
Define $\overline{v}$ to be the maximum variation between two rounds, and $\oV$ to be the maximum total variation. Then $\overline{v} = \max_{t \in \{1,2,\ldots,T \}} \left\|  \vv_t \right\|$, and $V_T \leq \overline{V}$. 

\subsection{Online Newton update}
We first provide a lemma, which we use to derive bounds on the \texttt{ONM} regret.

\begin{lemma}
\label{lem:matrix_norm}{\cite[Section 2.7, Problem 7]{kreyszig1978introductory}}
Let $\MM \in \Re^{n \times n}$ be a symmetric matrix and $h > 0$. Then, 
\[
\left\| \MM^{-1} \right\| \leq \frac{1}{h} \iff \left\| \MM \vv \right\| \geq h \left\|\vv \right\|\  \forall \; \vv \in \Re^{n}.
\]
\end{lemma}

We now use the previous lemma and derive the following identities regarding the Newton update.
\begin{lemma}[Online Newton update identities]
\label{lem:newton_update}
Suppose:
\begin{enumerate}
  \item $\exists h_t > 0$ such that $\left\| \HH_{t}^{-1} \left(\xs \right) \right\| \leq \frac{1}{h_t}$;
  \item $\exists \beta_t, L_t >0$ such that 
\end{enumerate}
\[
\left\| \xx - \xs \right\| \leq \beta_t \implies \left\| \HH_t\left(\xx\right) - \HH_t\left( \xs \right) \right\| \leq L_t \left\| \xx - \xs \right\|;
\]
\begin{enumerate}
\setcounter{enumi}{2}	
  \item $\left\| \xt - \xs \right\| \leq \gamma_t = \min\left\{\beta_t, \frac{2h_t}{3L_t} \right\}$. 
\end{enumerate}
Then, for the online Newton update, we have the following two identities:
\begin{align}
\left\| \xtt - \xs \right\| &< \left\| \xs - \xt \right\| \label{eq:second}\\
\left\| \xtt - \xs \right\| &\leq \frac{3L_t}{2h_t} \left\| \xs - \xt \right\|^2 \label{eq:third}
\end{align}

\end{lemma}

\begin{proof}
Consider the online Newton update given in Definition~\ref{def:newton_update}. Subtracting $\xs$ on both side leads to
\begin{align*}
\xtt - \xs &= \xt - \xs - \HH_{t}^{-1} \left(\xt \right) \nabla f_t \left( \xt \right)\\
&= \xt - \xs \\
&\quad- \HH_{t}^{-1} \left(\xt \right) \nabla f_t \left( \xt \right) + \HH_{t}^{-1} \left(\xt \right) \nabla f_t \left( \xs \right)\\
&= \xt - \xs + \HH_{t}^{-1} \left(\xt \right) \left( \nabla f_t \left( \xs \right) - \nabla f_t \left( \xt \right) \right).
\end{align*}
By the fundamental theorem of calculus, we have
\begin{align*}
\xtt - \xs &= \xt - \xs \ +\\ 
&\HH_{t}^{-1} \left(\xt \right) \int_0^1 \HH_{t} \left(\xt + \tau \left( \xs - \xt \right) \right) \left( \xs - \xt \right) \dx \tau.
\end{align*}
Consequently,
\begin{align*}
\xtt - \xs &= \HH_{t}^{-1}\left(\xt\right) \HH_{t}\left(\xt\right) \left(\xt - \xs\right) \ +\\
& \HH_{t}^{-1} \left(\xt \right) \int_0^1 \HH_{t} \left(\xt + \tau \left( \xs - \xt \right) \right) \left( \xs - \xt \right) \dx \tau\\
&= \HH_{t}^{-1}\left(\xt\right) \int_0^1 \HH_{t}\left(\xt\right) \left(\xt - \xs\right) \dx \tau \ +\\
& \HH_{t}^{-1} \left(\xt \right) \int_0^1 \HH_{t}\left(\xt + \tau \left( \xs - \xt \right) \right) \left( \xs - \xt \right) \dx \tau\\
&= \HH_{t}^{-1}\left(\xt\right) \int_0^1 \left(\HH_{t} \left(\xt + \tau \left( \xs - \xt \right) \right)\right.\\
&\qquad\qquad\qquad\qquad\qquad \left.- \HH_{t}\left(\xt\right)\right) \left( \xs - \xt \right) \dx \tau.
\end{align*}
Taking the norm of both sides, we have
\begin{align*}
\left\| \xtt - \xs \right\| &= \left\|\HH_{t}^{-1}\left(\xt\right) \int_0^1 \left(\HH_{t} \left(\xt + \tau \left( \xs - \xt \right) \right)\right.\right.\\
&\qquad\qquad\qquad\qquad \left.\left.- \HH_{t}\left(\xt\right)\right) \left( \xs - \xt \right) \dx \tau \vphantom{\int_0^1}\right\| \\
&\leq \left\|\HH_{t}^{-1}\left(\xt\right) \right\| \left\| \int_0^1 \left(\HH_{t} \left(\xt + \tau \left( \xs - \xt \right) \right)\right.\right.\\
&\qquad\qquad\qquad\qquad \left.\left.- \HH_{t}\left(\xt\right)\right) \left( \xs - \xt \right) \dx \tau \vphantom{\int_0^1}\right\|.
\end{align*}
The second and third assumptions allow us to upper bound the argument of the integral using Lipschitz continuity of the Hessian. This leads to
\begin{align}
\left\| \xtt - \xs \right\| &\leq \left\| \HH_{t}^{-1}\left(\xt\right) \right\| \int_0^1 \tau L_t  \left\| \xs - \xt \right\|^2 \dx \tau \nonumber\\
&= \left\| \HH_{t}^{-1}\left(\xt\right) \right\| \frac{ L_t}{2} \left\| \xs - \xt \right\|^2. \label{eq:first_bound}
\end{align}

We now derive a lower bound on the norm of the matrix product $\HH_{t}\left(\xt\right) \vv$ for all $\vv \in \Re^n$. We have
\begin{align*}
\left\| \HH_{t}\left(\xt\right) \vv \right\| &= \left\| \HH_{t}\left(\xt\right) \vv + \HH_{t}\left(\xs\right) \vv - \HH_{t}\left(\xs\right) \vv \right\| \\
&\geq \left\| \HH_{t}\left(\xs\right) \vv \right\|- \left\| \HH_{t}\left(\xt\right) \vv - \HH_{t}\left(\xs\right) \vv \right\|\\
&\geq \left\| \HH_{t}\left(\xs\right) \vv \right\| - \left\| \HH_{t}\left(\xt\right) - \HH_{t}\left(\xs\right)\right\| \left\| \vv \right\|,
\end{align*}
where we have used the reverse triangle inequality. By Lemma~\ref{lem:matrix_norm}, we have
\begin{align*}
\left\| \HH_{t}\left(\xt\right) \vv \right\| &\geq h_t \left\| \vv \right\|- L_t\left\| \xt - \xs \right\| \left\| \vv \right\| \\
&= \left(h_t - L_t\left\| \xt - \xs \right\| \right)\left\| \vv \right\|.
\end{align*}
We use the converse of Lemma~\ref{lem:matrix_norm} to obtain
\begin{equation}
\left\| \HH_{t}^{-1}\left(\xt\right)\right\| \leq \frac{1}{h_t - L_t\left\| \xt - \xs \right\|}.
\label{eq:bound_h1}
\end{equation}
Substituting~\eqref{eq:bound_h1} in~\eqref{eq:first_bound}, we obtain:
\begin{equation}
\left\| \xtt - \xs \right\| \leq \frac{L_t}{2\left(h_t - L_t\left\| \xt - \xs \right\|\right)} \left\| \xs - \xt \right\|^2. 
\label{eq:first_result}
\end{equation}
By assumption, $\left\| \xtt - \xs \right\| < \frac{2h}{3L_t}$ , and therefore~\eqref{eq:first_result} can be bounded above as follows.
\begin{align*}
\left\| \xtt - \xs \right\| &\leq \frac{L_t \left\| \xs - \xt \right\|}{2\left(h_t - L_t\left\| \xt - \xs \right\|\right)} \left\| \xs - \xt \right\|\\
&< \frac{L_t \frac{2h_t}{3L_t}}{2\left(h_t - L_t \frac{2h_t}{3L_t}\right)} \left\| \xs - \xt \right\|.
\end{align*}
Hence, we obtain
\begin{equation}
\left\| \xtt - \xs \right\| < \left\| \xs - \xt \right\|,
\end{equation}
which is the first identity.
Bounding the denominator of~\eqref{eq:first_result} yields the second identity:
\begin{align*}
\left\| \xtt - \xs \right\| &\leq \frac{L_t}{2\left(h_t - L_t \frac{2h}{3L_t}\right)} \left\| \xs - \xt \right\|^2 \\
&= \frac{3L_t}{2h_t} \left\| \xs - \xt \right\|^2.
\end{align*}
This completes the proof.
\end{proof}

Newton's method and our online extension can only be used if the Hessian is invertible along the algorithm's path, and by construction, also invertible at the point the algorithm converges to, $\xs$. For minimization, this translates to a function that is locally strongly convex in a ball of radius $\epsilon_t$ around $\xs$ for all $t$ (Assumption 1 of Lemma~\ref{lem:newton_update}). Without loss of generality assume that $\beta_t=\gamma_t= 2h_t/3L_t$.  For the Newton’s method to converge, it is enough for it to be initialized in a ball of radius $\gamma_t$ around $\xs$ (Assumption 3 of Lemma~\ref{lem:newton_update}). Note that the $\gamma_t$-ball does not need to be a subset of the aforementioned $\epsilon_t$-ball~\cite[Section 1.4]{bertsekas1999nonlinear}. Newton’s method converges to the solution even if it is initialized at a point outside the set where the Hessian is positive definite. The Hessian of the loss function $f_t$ is thus not necessarily positive definite for all $\xt$, $t=1,2,\ldots, T$. One may extend the results presented in this manuscript for the cases where the Hessian is not invertible by introducing modifications of the type discussed in~\cite[Section 3.4]{nocedal2006numerical}.

In brief, the previous result shows that if the Hessian is nonsingular and locally Lipschitz at $\xs$ and the initial point $\xt$ is close enough to $\xs$, then the online Newton update at $t$ moves strictly closer to $\xs$. Second, it shows that the update approaches $\xs$ at a quadratic rate.

The next result specifies a sufficient condition under which the assumptions of Lemma~\ref{lem:newton_update} at round $t$ will also be satisfied for round $t+1$. The next results can be derived with time-dependent parameters, but we use constant parameters to simplify notation.

\begin{lemma}
Suppose $\left\| \xt - \xs \right\| \leq \gamma$ and $\gamma \in \left( 0 , \frac{2h}{3L}\right)$. If the online Newton update is used and $\overline{v} \leq \gamma - \frac{3L}{2h}\gamma^2$, then $\left\| \xx_{t+1} - \xx_{t+1}^\ast \right\| < \gamma$.
\label{lem:always_sufficient}
\end{lemma}

\begin{proof}
We first re-express the initial sufficient condition for the online Newton update at $t+1$ using~\eqref{def:v}. We have
\begin{align*}
\left\| \xx_{t+1} - \xx_{t+1}^\ast \right\| &= \left\| \xx_{t+1} - \xs + \vv_t \right\| \\
&\leq \left\| \xx_{t+1} - \xs \right\| + \left\| \vv_t \right\| \\
&\leq \frac{3L}{2h} \left\| \xs - \xt \right\|^2 + \left\| \vv_t \right\|,
\end{align*}
where we used the identity~\eqref{eq:third} of Lemma~\ref{lem:newton_update} to obtain the last bound. Thus,
\begin{equation*}
\frac{3L}{2h} \left\| \xs - \xt \right\|^2 + \left\| \vv_t \right\| \leq \gamma \implies \left\| \xx_{t+1} - \xx_{t+1}^\ast \right\| < \gamma. \label{eq:bound_above}
\end{equation*}
We now upper bound $\frac{3L}{2h} \left\| \xs - \xt \right\|^2 + \left\| \vv_t \right\|$. By assumption we have $\left\| \xt - \xs \right\| \leq \gamma$:
\begin{equation*}
\frac{3L}{2h} \left\| \xs - \xt \right\|^2 + \left\| \vv_t \right\| \leq \frac{3L}{2h} \gamma^2 + \left\| \vv_t \right\|.
\end{equation*}
Upper bounding $\left\|\vv_t\right\|$ yields
\begin{equation}
\frac{3L}{2h} \gamma^2 + \left\| \vv_t \right\| \leq \frac{3L}{2h} \gamma^2 + \overline{v} .\label{eq:rhs_impose}
\end{equation}
If we assume that the right-hand term of~\eqref{eq:rhs_impose} can be bounded above by $\gamma$, we obtain
\[
\frac{3L}{2h} \gamma^2 + \overline{v} \leq \gamma \implies \left\| \xx_{t+1} - \xx_{t+1}^\ast \right\| < \gamma.
\]
Hence, if $\overline{v} \leq \gamma - \frac{3L}{2h} \gamma^2$ for $\gamma \in \left( 0 , \frac{2h}{3L}\right)$, then $\left\| \xx_{t+1} - \xx_{t+1}^\ast \right\| < \gamma$.
\end{proof}
Thus, if the initial decision $\xx_t$ is close enough to $\xs$ and the variation $\ov \leq \gamma - \frac{3L}{2h} \gamma^2$, Lemma~\ref{lem:newton_update} can be used sequentially. This is formalized in the next section.

\section{Online Newton's Method}
\label{sec:onm}
We now present \texttt{ONM}, described in Algorithm~\ref{alg:onm}, for online nonconvex optimization.

\begin{algorithm}[h]
\begin{algorithmic}[1]
\STATEx \textbf{Parameters:} Given $h$, $L$, and $\beta$. 
\STATEx \textbf{Initialization:} Receive $\xx_0 \in \Re$ such that $\left\|\xx_0 - \xx_0^\ast\right\| < \gamma = \min\left\{\beta, \frac{2h}{3L} \right\}$.
\medskip

\FOR{$t = 0, 1,2, \ldots, T$}
\STATE Play the decision $\xx_t$.
\STATE Observe the outcome at $t$: $f_t(\xx_t)$.
\STATE Update the decision:
\[
\xtt = \xt - \HH_{t}^{-1} \left(\xt \right) \nabla f_t \left( \xt \right).
\]
\ENDFOR
\end{algorithmic}
\caption{Online Newton's Method (\texttt{ONM})}
\label{alg:onm}
\end{algorithm}

First, we present an $O\left(V_T+1\right)$ general regret bound for \texttt{ONM} in Theorem~\ref{th:newton_method} and then discuss its implication.

\begin{theorem}[Regret bound for \texttt{ONM}]
Suppose that
\begin{enumerate}
  \item $\exists h > 0$ such that $\left\| \HH_{t}^{-1} \left(\xs \right) \right\| \leq \frac{1}{h}$ for all $t=1,2,\ldots,T$;
  \item $\exists \beta, L >0$ such that 
  $\left\| \xx - \xs \right\| \leq \beta \implies \left\| \HH_t\left(\xx\right) - \HH_t\left( \xs \right) \right\| \leq L \left\| \xx - \xs \right\|;$
  \item $\exists \xx_0 \in \Re^n$ such that $\left\| \xx_0 - \xx^\ast_0 \right\| \leq \gamma = \min\left\{\beta, \frac{2h}{3L} \right\}$;
  \item $\overline{v} \leq \gamma - \frac{3L}{2h}\gamma^2$;
  \item $\exists \ell > 0$ such that $\left\| \xx - \xs \right\| \leq \gamma \implies \left| f_t(\xx) - f_t(\xs) \right| \leq \ell \left\| \xx - \xs \right\|$ for all $t=1,2,\ldots,T$.
\end{enumerate}
Then, the regret of \texttt{ONM} is bounded above by
\[
\reg(T) \leq \frac{\ell}{1 - \frac{3L}{2h}\gamma}\left( V_T + \delta \right),
\]
where $\delta=  \frac{3L}{2h}\left(\left\| \xx_0 - \xx^\ast_0 \right\|^2 - \left\| \xx^\ast_T - \xx_T \right\|^2\right)$. Equivalently, $\reg(T) \leq O\left(V_T+1\right)$ and is sublinear in the number of rounds if $V_T < O \left(T\right)$.
\label{th:newton_method}
\end{theorem}

\begin{proof}
First, observe that the third assumption of Lemma~\ref{lem:newton_update} is always met after an online Newton update for $\overline{v} \leq \gamma - \frac{3L}{2h} \gamma^2$, and that $\gamma \in \left( 0 , \frac{2h}{3L}\right)$ by Lemma~\ref{lem:always_sufficient}. 

Next, we obtain an upper bound on the sum of $\| \xt - \xs \|$ over $t=1,2,\ldots, T$ by adapting~\cite{mokhtari2016online}'s proof technique. We then use this result to upper bound the regret of \texttt{ONM}. For $t \geq 1$, we have
\begin{align}
\left\| \xt - \xs \right\| &= \left\| \xt - \xs - \xtm^\ast + \xtm^\ast  \right\| \nonumber\\
&\leq \left\| \xt - \xtm^\ast \right\| +  \left\| \xtm^\ast  - \xs \right\| \label{eq:used_tri}\\
&\leq \frac{3L}{2h}\left\| \xx^\ast_{t-1} - \xx_{t-1} \right\|^2 +  \left\| \xtm^\ast  - \xs \right\|, \label{eq:used_result3}
\end{align}
where we first used the triangle inequality to obtain~\eqref{eq:used_tri} and then identity~\eqref{eq:third} of Lemma~\ref{lem:newton_update} for~\eqref{eq:used_result3}. Summing~\eqref{eq:used_result3} from $t=1$ to $T$, we obtain
\begin{align*}
\sum_{t=1}^T \left\| \xt - \xs \right\| &\leq \sum_{t=1}^T  \frac{3L}{2h}\left\| \xx^\ast_{t-1} - \xx_{t-1} \right\|^2 + \left\| \xtm^\ast  - \xs \right\| \\
&=\sum_{t=1}^T \frac{3L}{2h}\left\| \xs - \xt \right\|^2  + \sum_{t=1}^T \left\| \xtm^\ast  - \xs \right\|\\
&\quad + \frac{3L}{2h}\left\| \xx_0 - \xx^\ast_0 \right\|^2 - \frac{3L}{2h}\left\| \xx^\ast_T - \xx_T \right\|^2. 
\end{align*}
Thus, we have
\begin{equation}
\sum_{t=1}^T \left\| \xt - \xs \right\| \leq \sum_{t=1}^T \frac{3L}{2h}\left\| \xs - \xt \right\|^2 + V_T + \delta, \label{eq:to_rex}
\end{equation}
where we used the definitions: $V_T = \sum_{t=1}^T \left\| \xtm^\ast  - \xs \right\|$ and $\delta = \frac{3L}{2h}\left(\left\| \xx_0 - \xx^\ast_0 \right\|^2 - \left\| \xx^\ast_T - \xx_T \right\|^2\right)$. Before solving for $\left\| \xs - \xt \right\|$, we re-express~\eqref{eq:to_rex} as:
\begin{equation}
\sum_{t=1}^T \left( \left\| \xt - \xs \right\| - \frac{3L}{2h}\left\| \xt - \xs \right\|^2\right) \leq V_T +\delta. \label{eq:will_call_lemma}
\end{equation}
Using Lemma~\ref{lem:always_sufficient}, we can lower bound $\left\| \xt - \xs \right\|^2$ by $\gamma\left\| \xt - \xs \right\|$ in the left-hand side of~\eqref{eq:will_call_lemma}. We then obtain
\begin{equation*}
\sum_{t=1}^T \left( \left\| \xt - \xs \right\| - \frac{3L}{2h}\gamma\left\| \xt - \xs \right\|\right) \leq V_T +\delta,
\end{equation*}
and equivalently,
\begin{equation*}
\sum_{t=1}^T \left\| \xt - \xs \right\| \left( 1 - \frac{3L}{2h}\gamma \right) \leq  V_T +\delta. \label{eq:reused}
\end{equation*}
where $\left( 1 - \frac{3L}{2h}\gamma \right) > 0$ because $\gamma \in \left(0,\frac{2h}{3L}\right)$. The sum of errors is thus bounded by
\begin{equation}
\sum_{t=1}^T \left\| \xt - \xs \right\| \leq \left( 1 - \frac{3L}{2h}\gamma \right)^{-1}\left( V_T + \delta \right). \label{eq:bound_se}
\end{equation}
Now, recall the definition of the regret:
\[
\reg(T) = \sum_{t=1}^T f_t(\xt) - f_t(\xs).
\]
By assumption, $f_t$ is locally $\ell$-Lipschitz for all $\xs$. The regret can be bounded above by
\begin{equation}
\reg(T) \leq \sum_{t=1}^T \ell \left\| \xt - \xs \right\|. \label{eq:lip_regret}
\end{equation}
Substituting~\eqref{eq:bound_se} in~\eqref{eq:lip_regret} yields
\[
\reg(T) \leq \frac{\ell}{1 - \frac{3L}{2h}\gamma}\left( V_T + \delta \right),
\]
which completes the proof.
\end{proof}

By Theorem~\ref{th:newton_method}, \texttt{ONM} regret is bounded above by $O\left(V_T+1\right)$. The regret of $\sigma$\texttt{OGD} for strongly convex functions has an upper bound of the same order~\cite{mokhtari2016online}. Similar to $\sigma$\texttt{OGD}, our approach requires $V_T < O\left( T \right)$ to have a sublinear regret. An advantage of our approach is that it does not require (strong-)convexity or Lipschitz continuous gradients. Assumption~1 of Theorem~\ref{th:newton_method} requires the Hessian to be nonsingular at $\xx_t^*$. It does not require $f_t$ to be strongly-convex nor even convex. An example of such an objective function is the least-squares range measurement problem for localizing a moving target by an array of sensors. This problem is nonconvex, but at its optimum, the Hessian is nonsingular and Lipschitz continuous. This example will be explored in detail in Section~\ref{sec:example}. Other potential applications of \texttt{ONM} are concave-convex games~\cite[Section 10.3.4]{boyd2004convex}. We note that \texttt{ONM} can be applied to track any kind of stationary point, e.g., a saddle point arising from minimax problems.

\texttt{OGD} and \texttt{DMD} for convex functions both have $O\left(\sqrt{T}\left(V_T + 1\right)\right)$ regret bounds~\cite{zinkevich2003online,hall2015online}. This requires $V_T < O\left( \sqrt{T} \right)$ and convexity for sublinear regret. The example presented in Section~\ref{sec:example} emphasizes the importance of our result as standard OCO cannot offer a performance guarantee on nonconvex loss functions and may perform poorly. \texttt{OGD}, $\sigma$\texttt{OGD} and \texttt{DMD} do, however, have extensions to handle time-invariant constraints in the dynamic regret setting~\cite{neely2017online,chen2017online,cao2018virtual,paternain2017online,tang2018feedback,lesage2018online}. 

\begin{remark}
Convexity is not required for \texttt{ONM} to attain optimality. Assuming local Lipschitz properties, if $\left\|\xx_0 - \xx_0^\ast\right\| < \gamma$ and variations are bounded, the decision $\xx_t$ remains in the neighbourhood of the optimum regardless of the convexity of the loss function.
\end{remark}

We now show that the regret bound of \texttt{ONM} reduces to a constant when the total variation is bounded. We first present a lemma about the convergence of a quadratic map, and then proceed to the regret bound.

\begin{lemma}[Quadratic map convergence]
Let $x_{n+1} = c x_n^2 + v$ for $n \in \mathbb{N}$ and where $c > 0$ and $v > 0$. If $x_0 \in \left[ 0 , \overline{x}\right)$ and $v \leq \frac{1}{4c}$, then $\left( x_n \right) \to \ux$ where
\begin{align*}
\ox &= \frac{1 + \sqrt{1-4cv}}{2c},\\ 
\ux &= \frac{1 - \sqrt{1-4cv}}{2c}.
\end{align*}
\label{lem:quad_map_c}
\end{lemma}
The proof is provided in the Appendix. Let
\begin{align}
\oE&= \frac{h}{3L}\left(1 + \sqrt{1-\frac{6L\left(\oV + \left\| \xx_0 - \xx^\ast_0 \right\| \right)}{h}}\right),  \label{eq:oE} \\
\uE &= \frac{h}{3L}\left(1 - \sqrt{1-\frac{6L\left(\oV + \left\| \xx_0 - \xx^\ast_0 \right\| \right)}{h}}\right). \label{eq:uE}
\end{align}

\begin{corollary}[Constant regret bound of \texttt{ONM}]
Suppose all the assumptions of Theorem~\ref{th:newton_method} are met. If $\gamma  < \oE$ and $\oV + \left\| \xx_0 - \xx_0^* \right\| \leq \frac{h}{6L}$, then the regret of \texttt{ONM} is upper bounded by
\[
\reg(T) \leq \ell \uE.
\]
Equivalently, we have $\reg(T) \leq O\left( 1\right)$.
\label{thm:o1}
\end{corollary}

\begin{proof}
Consider the error at round $t$, $\left\| \xt - \xs \right\|$. By the triangle inequality, 
\begin{align*}
\left\| \xt - \xs \right\| &= \left\| \xt - \xs - \xtm^\ast + \xtm^\ast  \right\| \nonumber\\
&\leq \left\| \xt - \xtm^\ast \right\| +  \left\| \xtm^\ast  - \xs \right\|.
\end{align*}
Using identity~\eqref{eq:third} of Lemma~\ref{lem:newton_update}, we have
\begin{equation}
\left\| \xt - \xs \right\|\leq c\left\| \xx^\ast_{t-1} - \xx_{t-1} \right\|^2 +  \left\| \xtm^\ast  - \xs \right\|, \label{eq:to_seq}
\end{equation}
where $c=\frac{3L}{2h}$. Let $e_t = \left\| \xt - \xs \right\|$ and recall that $v_t = \left\| \xtt^\ast  - \xs \right\|$. We rewrite~\eqref{eq:to_seq} as
\[
e_t \leq c e_{t-1}^2 + v_{t-1}.
\]
Summing through time yields
\[
\sum_{t=1}^T e_t \leq c\sum_{t=1}^T e_{t-1}^2 + \sum_{t=1}^T v_{t-1},
\]
or equivalently,
\[
\sum_{t=0}^T e_t \leq c\sum_{t=0}^{T-1} e_{t}^2 + \sum_{t=0}^{T-1} v_t + e_0.
\]
The sum of squares is bounded above by the square of the sum, and thus
\begin{equation}
\sum_{t=0}^T e_t \leq c\left(\sum_{t=0}^{T-1} e_{t}\right)^2 + V_T + e_0, \label{eq:to_Et}
\end{equation}
where we also used the definition $V_T = \sum_{t=0}^{T-1} v_t$. Let $E_T = \sum_{t=0}^T e_t$. We re-express~\eqref{eq:to_Et} as
\[
E_T \leq c E_{T-1}^2 + V_T + e_0.
\]
Finally, we have
\begin{equation}
E_T \leq c E_{T-1}^2 + \oV + e_0, \label{eq:E_recurrence}
\end{equation}
where $V_T = \sum_{t=0}^{T-1} \left\| \xtt^\ast  - \xs \right\| \leq \oV$. Note that we obtain $\overline{E}$ and $\underline{E}$ by applying Lemma~\ref{lem:quad_map_c} to the sequence obtained when~\eqref{eq:E_recurrence} is met with equality.

The regret is bounded above by
\begin{align}
\reg(T) &= \sum_{t=1}^T f_t(\xt) - f_t(\xs) \nonumber\\
&\leq \sum_{t=0}^T \ell \left\| \xt - \xs \right\|, \label{eq:regret_rhs}
\end{align}
because $f_t$ is locally $\ell$-Lipschitz around $\xs$ for all $t$. The right-hand side of~\eqref{eq:regret_rhs} is a non-decreasing sum and thus
\[
\reg(T) \leq \ell \lim_{n \to \infty} \sum_{t=0}^{T+n} \left\| \xt - \xs \right\|,
\]
and equivalently,
\[
\reg(T) \leq \ell \lim_{n \to \infty} E_{T+n}.
\]
By assumption, $E_0 = \left\| \xx_0 - \xx^\ast_0 \right\| < \oE$ and $\oV + \left\| \xx_0 - \xx_0^* \right\| \leq \frac{h}{6L}$. Therefore Lemma~\ref{lem:quad_map_c} yields
\[
\reg(T) \leq \ell \uE,
\]
which completes the proof.
\end{proof}
A constant bound is desirable because it implies that (i) the total error by the algorithm is bounded and (ii) the average regret will converge more rapidly to $0$ than a bound that depends on $T$ or $V_T$.

The constant bound provided by Corollary~\ref{thm:o1} is strictly tighter than the bound in Theorem~1 when, in addition to Corollary~\ref{thm:o1}'s assumptions, we have
\[
\uE (1-y) < V_T +\delta,
\]
for $y \in \left(0, \frac{3L}{2h}\oE \right)$. For example, if $\left\| \xx_0 - \xx_0^* \right\|= \left\| \xx_T - \xx_T^* \right\|=0$, i.e., the initial optimum is known and the decisions made by the algorithm approximately converged after $T$ rounds given a fixed total variation, then the constant bound is strictly less than Theorem~1's bound if
\[
\uE (1-y) < V_T \leq \frac{h}{6L},
\]
for $y \in \left(0, \frac{3L}{2h}\oE \right)$. We remark that there always exists $\overline{y} \in \left(0, \frac{3L}{2h}\oE \right)$ such that $\uE (1-\overline{y}) < \frac{h}{6L}$.

Lastly, a tighter regret bound implies that the total maximum error made by the algorithm is smaller. It thus provides a better representation of the algorithm's actual performance.

\section{Example}
\label{sec:example}
In this section, we evaluate the performance of \texttt{ONM} in a numerical example based on localizing a moving target~\cite{cao2006sensor}. Let $\xx_t \in \Re^2$ be the location of the target at time $t$, and $\aa_i \in \Re^2$ be the position of sensor $i$. The range measurement is given by
\[
d_i = \left\| \xx_t - \aa_i \right\| + w_i,
\]
for $i=1,2,\ldots m$, where $m$ is the number of sensors and $w_i$ models noise. Localizing a target based on range measurements leads to the following nonlinear, nonconvex least-squares problem:
\begin{equation}
\min_{\xx_t \in \Re^n} \sum_{i=1}^m \left( \left\| \xx_t - \aa_i \right\| - d_i \right)^2, \label{eq:loss_func}
\end{equation}
for $t=1,2,\dots, T$. The objective is to sequentially estimate the position of the target while only having access to past information. We use \texttt{ONM} to track the target's location in real-time. No other online convex optimization algorithms can guarantee performance on nonconvex loss functions like~\eqref{eq:loss_func}.

We compare \texttt{ONM} with \texttt{OGD}~\cite{zinkevich2003online}. As previously mentioned, \texttt{OGD} does not guarantee adequate performance, but fits the sequential nature of the problem. 

Let $\xs$ be the actual position of the target at time $t$. The target's position evolves as $\xx_{t+1}^\ast = \xs + \vv_t$, where we describe $\vv_t$ in the next subsections. We set $\xx_0^\ast = \begin{pmatrix} 2 & 1 \end{pmatrix}\Tr$ and assume $\xx_0^\ast$ is known, i.e., $\xx_0 = \xx_0^\ast$. We consider three sensors. The sensors are located at $\aa_1=\begin{pmatrix} \frac{1}{2} & \frac{1}{2} \end{pmatrix}\Tr$, $\aa_2 = \begin{pmatrix} 0 & \frac{1}{2} \end{pmatrix}\Tr$, and $\aa_3 = \begin{pmatrix} \frac{1}{2} & 0 \end{pmatrix}\Tr$. 
Each sensor $i$ produces a range measurement, $d_i$, which is corrupted by Gaussian noise $w\sim \mathrm{N}(0,\sigma_w)$. For all simulations, we set $\gamma = \frac{h}{3L}$ and \texttt{OGD}'s step size to $\eta = 1/\sqrt{T}$.

\subsection{Moving target localization with \texttt{ONM}}
We first evaluate the performance of \texttt{ONM} in the general case. We set $\vv_t = \frac{(-1)^{b_t} 0.0025}{\sqrt{2t}} \mathbf{1}$ where $b_t \sim \textrm{Bernoulli}(0.5)$ and let $\sigma_{w} = 0.01\%$. All sufficient conditions of Theorem~\ref{th:newton_method} are satisfied. Figure~\ref{fig:regret_tracking} presents the experimental regret of \texttt{ONM} and \texttt{OGD} for localizing a moving target averaged over $1000$ simulations. The regret is sublinear in both cases, but \texttt{ONM} attains a much lower regret than \texttt{OGD}. The experimental regret is bounded above by Theorem~\ref{th:newton_method}'s bound. The bound is not shown in Figure~\ref{fig:regret_tracking} because it is several order of magnitude larger than the experimental regret.

\begin{figure}[!t]
  \centering
  \includegraphics[width=1\columnwidth]{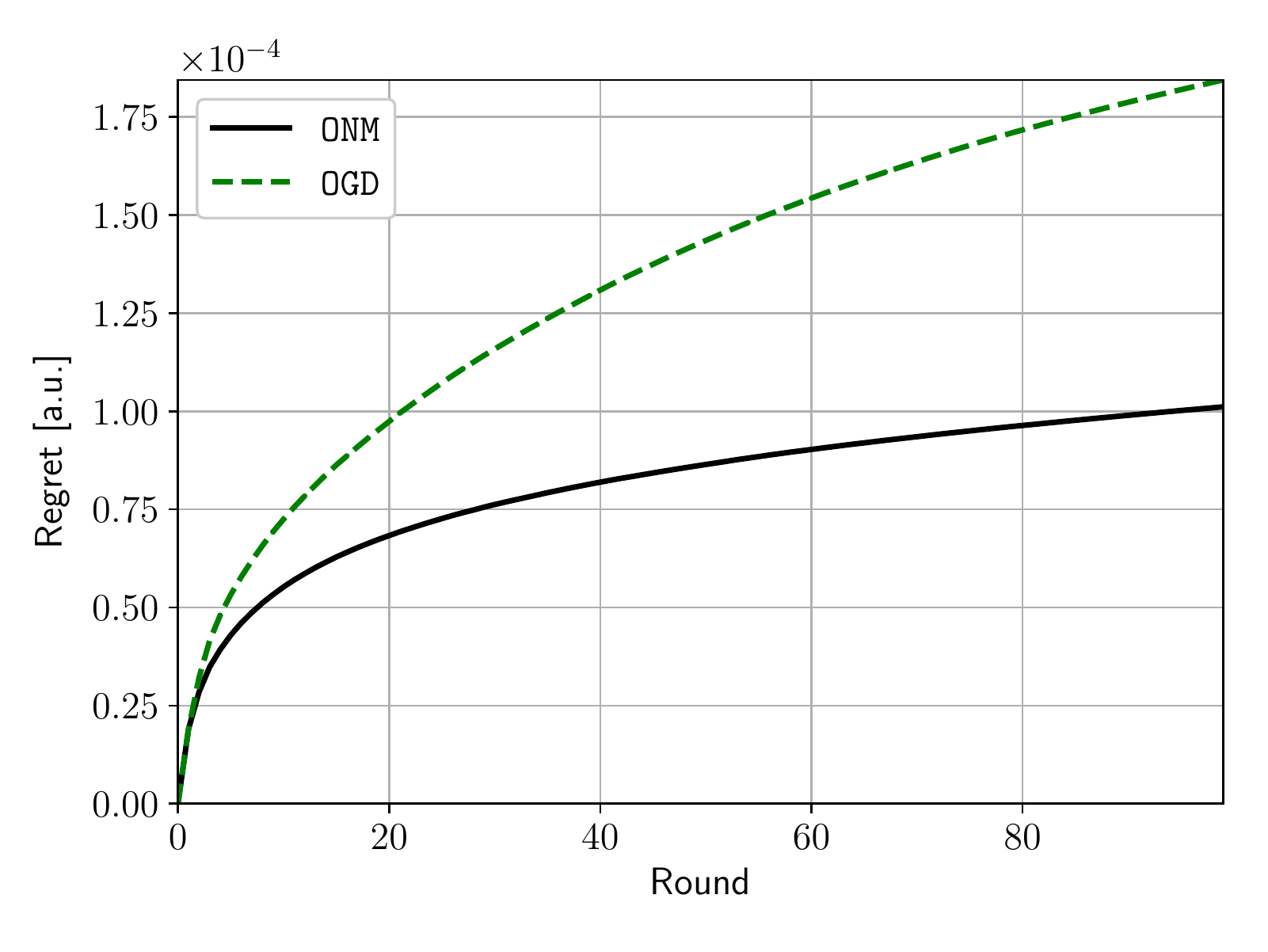}
  \vspace{-0.85cm} 
  \caption{Regret of moving target localization with \texttt{ONM} (averaged out of 1000 simulations)}
  \label{fig:regret_tracking}
\end{figure}

The tracking performance is presented in Figure~\ref{fig:target}. \texttt{OGD} effectively fails to follow the target. \texttt{ONM}, which is guaranteed to stay in the neighborhood of the target, remains accurate at every time step.

\begin{figure}[!t]
  \centering
  \includegraphics[width=1\columnwidth]{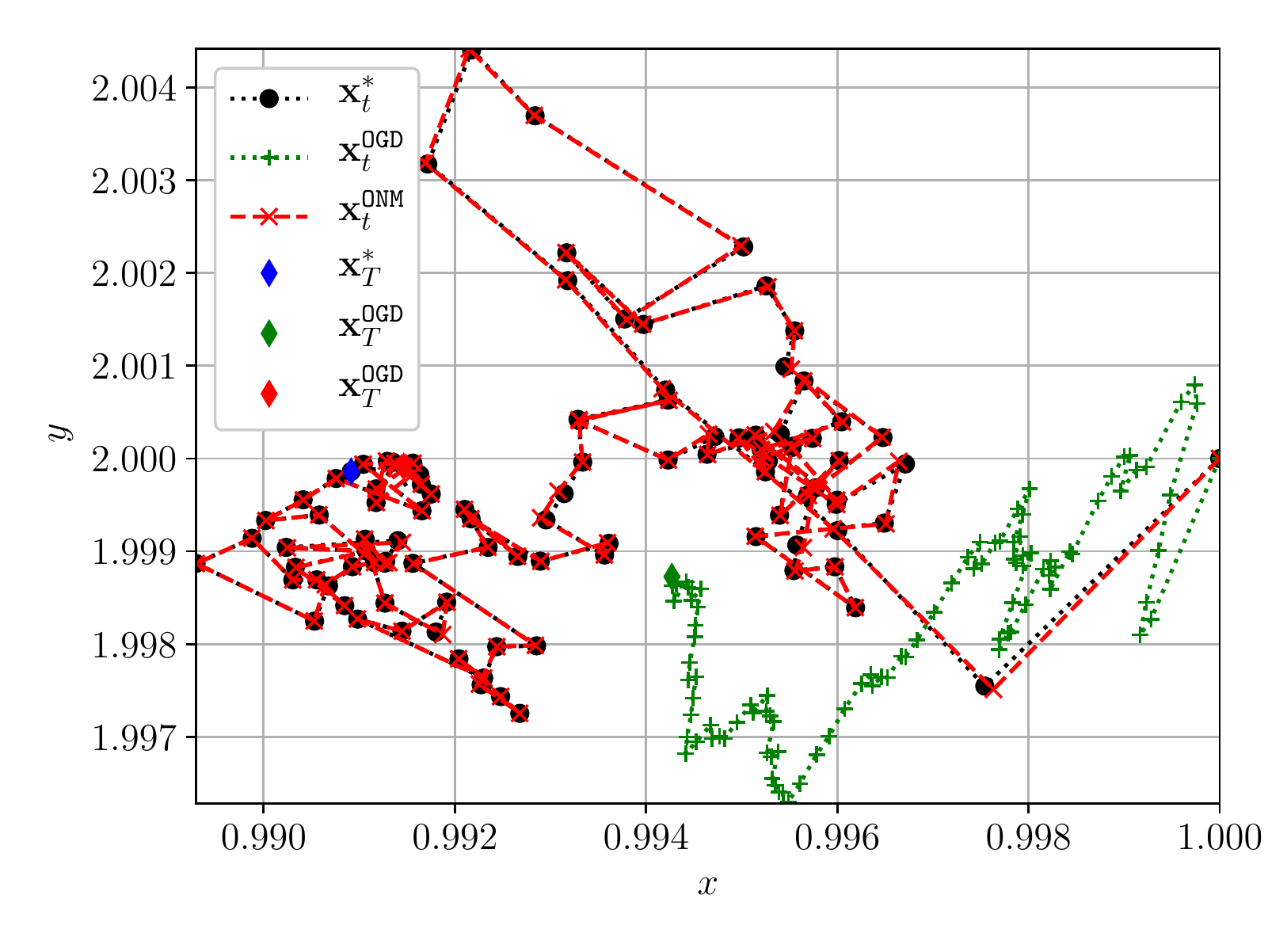}
  \vspace{-0.85cm} 
  \caption{Localization of a moving target with \texttt{ONM}}
  \label{fig:target}
\end{figure}

\subsection{Constant regret bound}
We test \texttt{OMN} when there is limited variation in the target's location. This case illustrates the $O\left(1\right)$ regret bound of Corollary~\ref{thm:o1}. We set $\ss_t = \frac{6 (-1)^b \overline{V}}{\sqrt{2} t^2 \pi^2} \mathbf{1}$, so that $\sum_{t=1}^T \left\| \ss_t \right\| < \sum_{t=1}^\infty \left\| \ss_t \right\| = \overline{V}$, and set $\overline{V}$ according to the sufficient condition of Corollary~\ref{thm:o1}. Lastly, we let $\sigma_w = 1 \times 10^{-6}$. Figure~\ref{fig:regret_tracking_o1} shows the constant regret for \texttt{ONM}. As in the previous example, \texttt{ONM}'s regret is lower than \texttt{OGD}'s. Figure~\ref{fig:regret_tracking_o1} shows the tracking performance. \texttt{ONM} outperforms \texttt{OGD}, which never reaches the vicinity of the target's position. The zoomed-in section of Figure~\ref{fig:target_o1} presents the last ten rounds of the simulations. Here we see that \texttt{ONM} achieves virtually perfect performance when measurement noise is low and $\vv_t$ extremely small.

\begin{figure}[!t]
  \centering
  \includegraphics[width=1\columnwidth]{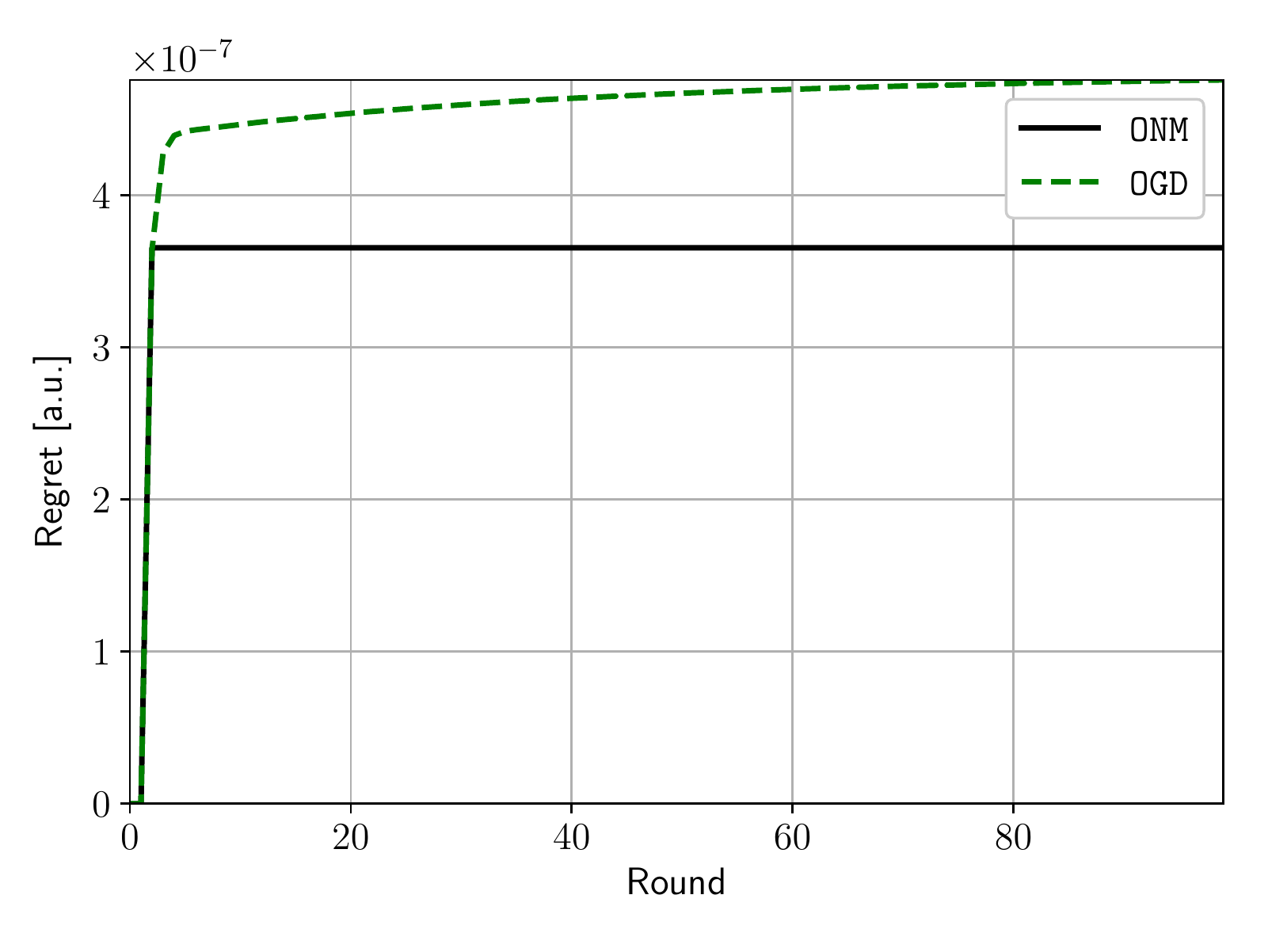}
  \vspace{-0.85cm} 
  \caption{Regret of moving target localization with \texttt{ONM} for limited total variation}
  \label{fig:regret_tracking_o1}
\end{figure}

\begin{figure}[!t]
  \centering
  \includegraphics[width=1\columnwidth]{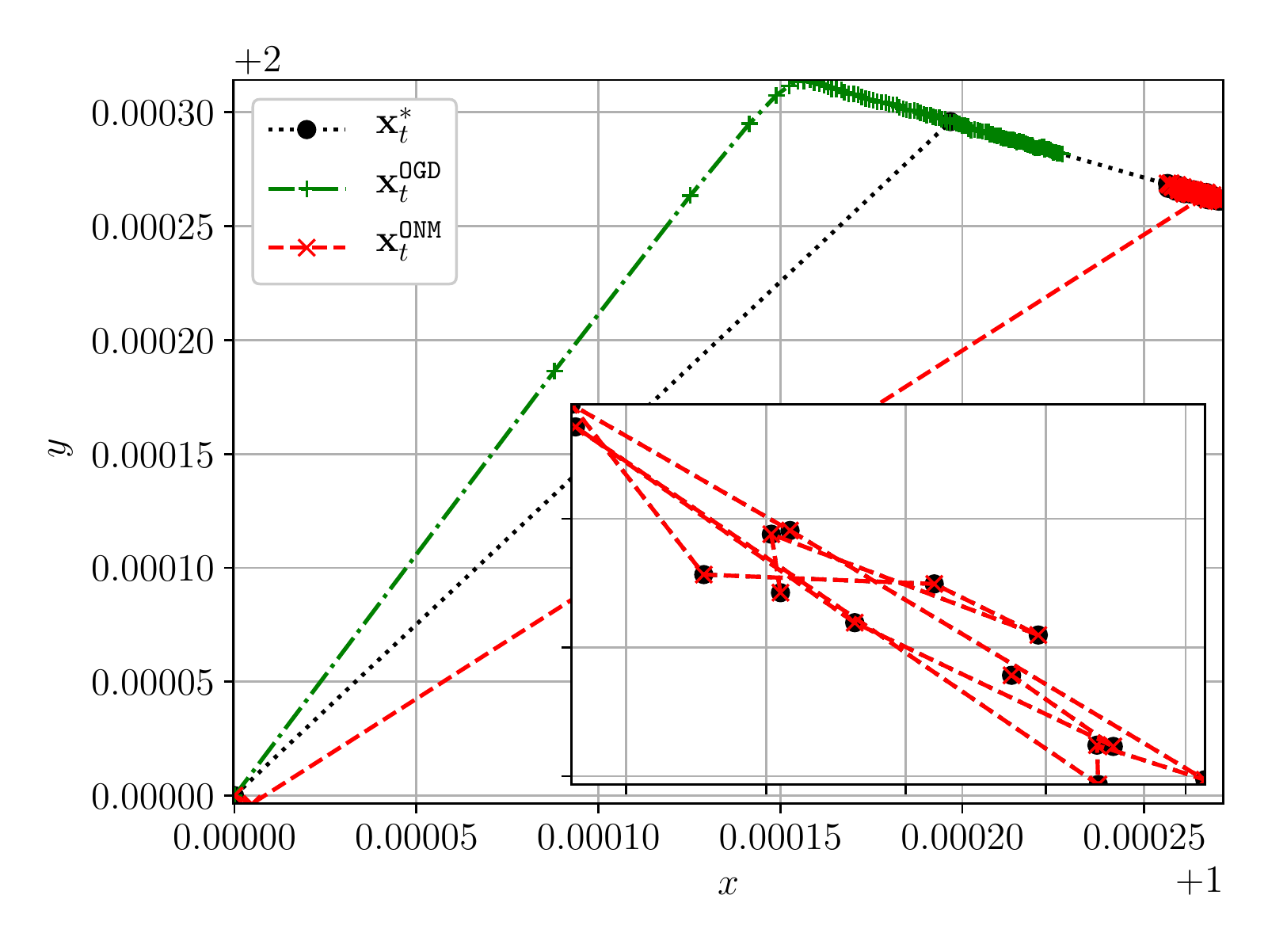}
  \vspace{-0.85cm} 
  \caption{Localizing a moving target with \texttt{ONM} for limited total variation (zoomed section: 10 last rounds)}
  \label{fig:target_o1}
\end{figure}

\section{Conclusion}
\label{sec:con}

In this work, we presented a second-order approach, the online Newton's method (\texttt{ONM}), for online nonconvex optimization problems. We only assume local Lipschitz properties on the objective function and bounded variations between the round optima. We provide two regret bounds: an $O\left(V_T+1\right)$ regret bound for the general case and a constant $O\left(1\right)$ regret bound for when the total variation between optima is limited. The first bound is similar to results in the literature but with an important difference: it also holds for nonconvex problems. We conclude with a moving target localization example using \texttt{ONM}. We show that our approach properly tracks a moving target using noisy range measurements and outperforms online gradient descent. In future work, we will investigate tighter regret bounds for \texttt{ONM} by taking advantage of the quadratic convergence rate of the Newton's update shown in Lemma~\ref{lem:newton_update}. Another possible future direction is combining first-order and second-order updates in a hybrid framework, which could result in better regret bounds under less stringent requirements.

\appendix[Proof of Lemma~\ref{lem:quad_map_c}]
\label{app:quad}
We split the initial condition of the sequence $(x_n)$, $x_0 \in [0,\overline{x})$, into two subsets: (i) $x_0 \in \left[\ux, \ox \right)$, and (ii) $x_0 \in \left[0, \ux\right]$.

\paragraph*{Case (i) $x_0 \in \left[\ux, \ox \right)$} First, we show that $x_n$ is strictly decreasing or equivalently, $x_{n+1} <  x_n$. This condition is equivalent to $c x_{n}^2 - x_n + v < 0$,
which holds for all $x_n \in \left( \ux , \ox \right)$ and $v\leq \frac{1}{4c}$. This further implies that $x_n<  \ox$ if $x_0 \in \left[\ux, \ox \right)$. As  result, $x_n$ never converges to $\ox$ even though $\ox$ is a fixed point of the map on the left-hand-side of the aforementioned inequality. Moreover, letting $x_n = \ux$ for any $n\in \NN$, we have
\begin{align*}
x_{n+1}&= c \ux^2 + v\\
&= c \left(\frac{1 - \sqrt{1-4cv}}{2c} \right)^2 + v\\
&= \frac{1 - \sqrt{1-4cv}}{2c}= \ux.
\end{align*}
Thus, $x_n > x_{n+1}$ and $\ux \leq x_{n} < \ox$ for all $n \in \NN$  assuming $x_0 \in \left[\ox ,\ux\right)$. The infimum, $\inf_{n \in \NN} x_n  = \ux$, is attained because of the strict monotonicity and $\ux$ being a fixed point of the map. Now, by the monotone convergence theorem, a monotonically decreasing and bounded below sequence converges and it converges to its infimum. Thus, $(x_n) \to \ux$ for any $x_0 \in \left[\ux, \ox\right)$ and $v \leq \frac{1}{4c}$.

\paragraph*{Case (ii) $x_0 \in \left[0, \ux \right]$} Similar to above, we show that the sequence is monotonically increasing, or equivalently, $x_n \leq x_{n+1}$. Similar to the previous case, this is equivalent to $0 \leq c x_n^2 - x_n + v$. 
The inequality holds strictly for all $x_n \in \left[0, \ux \right)$. Moreover, if $x_n = \ux$ for $n\in \NN$, then $x_{n+1} = \ux$ from the above demonstration and $\ux$ is a fixed point of the map. 

Hence, $x_n \in \left[0, \ux \right]$ for all $n \in \NN$ if $x_0 \in \left[0, \ux \right]$. 
The supremum, $\sup_{n\in \NN} x_n = \ux$, is reached for any initial condition because of the strict monotonicity and the fact that $\ux$ is a fixed point of the map.
Similarly to the previous case, by the monotone convergence theorem, a monotonically increasing and bounded above sequence converges and it converges to its supremum. Therefore, $(x_n) \to \ux$ for any $x_0 \in \left[0, \ux\right]$ and $v \leq \frac{1}{4c}$. Putting the results for two initial condition intervals together then proves the lemma.\hfill \QED



\begin{thebibliography}{10}
\providecommand{\url}[1]{#1}
\csname url@samestyle\endcsname
\providecommand{\newblock}{\relax}
\providecommand{\bibinfo}[2]{#2}
\providecommand{\BIBentrySTDinterwordspacing}{\spaceskip=0pt\relax}
\providecommand{\BIBentryALTinterwordstretchfactor}{4}
\providecommand{\BIBentryALTinterwordspacing}{\spaceskip=\fontdimen2\font plus
\BIBentryALTinterwordstretchfactor\fontdimen3\font minus
  \fontdimen4\font\relax}
\providecommand{\BIBforeignlanguage}[2]{{%
\expandafter\ifx\csname l@#1\endcsname\relax
\typeout{** WARNING: IEEEtran.bst: No hyphenation pattern has been}%
\typeout{** loaded for the language `#1'. Using the pattern for}%
\typeout{** the default language instead.}%
\else
\language=\csname l@#1\endcsname
\fi
#2}}
\providecommand{\BIBdecl}{\relax}
\BIBdecl

\bibitem{chen2017online}
T.~Chen, Q.~Ling, and G.~B. Giannakis, ``An online convex optimization approach
  to proactive network resource allocation,'' \emph{IEEE Transactions on Signal
  Processing}, vol.~65, no.~24, pp. 6350--6364, 2017.

\bibitem{cao2018virtual}
X.~Cao, J.~Zhang, and H.~V. Poor, ``A virtual-queue-based algorithm for
  constrained online convex optimization with applications to data center
  resource allocation,'' \emph{IEEE Journal of Selected Topics in Signal
  Processing}, vol.~12, no.~4, pp. 703--716, 2018.

\bibitem{ma2016distributed}
W.-J. Ma, V.~Gupta, and U.~Topcu, ``Distributed charging control of electric
  vehicles using online learning,'' \emph{IEEE Transactions on Automatic
  Control}, vol.~62, no.~10, pp. 5289--5295, 2016.

\bibitem{lesage2018setpoint}
A.~Lesage-Landry and J.~A. Taylor, ``Setpoint tracking with partially observed
  loads,'' \emph{IEEE Transactions on Power Systems}, vol.~33, no.~5, pp.
  5615--5627, 2018.

\bibitem{bedi2018tracking}
A.~S. Bedi, P.~Sarma, and K.~Rajawat, ``Tracking moving agents via inexact
  online gradient descent algorithm,'' \emph{IEEE Journal of Selected Topics in
  Signal Processing}, vol.~12, no.~1, pp. 202--217, 2018.

\bibitem{dixit2019online}
R.~Dixit, A.~S. Bedi, R.~Tripathi, and K.~Rajawat, ``Online learning with
  inexact proximal online gradient descent algorithms,'' \emph{IEEE
  Transactions on Signal Processing}, vol.~67, no.~5, pp. 1338--1352, 2019.

\bibitem{zinkevich2003online}
M.~Zinkevich, ``Online convex programming and generalized infinitesimal
  gradient ascent,'' in \emph{Proceedings of the 20th International Conference
  on Machine Learning (ICML-03)}, 2003, pp. 928--936.

\bibitem{shalev2012online}
S.~Shalev-Shwartz \emph{et~al.}, ``Online learning and online convex
  optimization,'' \emph{Foundations and Trends{\textregistered} in Machine
  Learning}, vol.~4, no.~2, pp. 107--194, 2012.

\bibitem{hazan2016introduction}
E.~Hazan \emph{et~al.}, ``Introduction to online convex optimization,''
  \emph{Foundations and Trends{\textregistered} in Optimization}, vol.~2, no.
  3-4, pp. 157--325, 2016.

\bibitem{mokhtari2016online}
A.~Mokhtari, S.~Shahrampour, A.~Jadbabaie, and A.~Ribeiro, ``Online
  optimization in dynamic environments: Improved regret rates for strongly
  convex problems,'' in \emph{Decision and Control (CDC), 2016 IEEE 55th
  Conference on}, 2016, pp. 7195--7201.

\bibitem{cao2006sensor}
M.~Cao, B.~D. Anderson, and A.~S. Morse, ``Sensor network localization with
  imprecise distances,'' \emph{Systems \& Control Letters}, vol.~55, no.~11,
  pp. 887--893, 2006.

\bibitem{zhang2017improved}
L.~Zhang, T.~Yang, J.~Yi, J.~Rong, and Z.-H. Zhou, ``Improved dynamic regret
  for non-degenerate functions,'' in \emph{Advances in Neural Information
  Processing Systems}, 2017, pp. 732--741.

\bibitem{gao2018online}
X.~Gao, X.~Li, and S.~Zhang, ``Online learning with non-convex losses and
  non-stationary regret,'' in \emph{International Conference on Artificial
  Intelligence and Statistics}, 2018, pp. 235--243.

\bibitem{hazan2017efficient}
E.~Hazan, K.~Singh, and C.~Zhang, ``Efficient regret minimization in non-convex
  games,'' \emph{arXiv preprint arXiv:1708.00075}, 2017.

\bibitem{hazan2007logarithmic}
E.~Hazan, A.~Agarwal, and S.~Kale, ``Logarithmic regret algorithms for online
  convex optimization,'' \emph{Machine Learning}, vol.~69, no. 2-3, pp.
  169--192, 2007.

\bibitem{hall2015online}
E.~C. Hall and R.~M. Willett, ``Online convex optimization in dynamic
  environments,'' \emph{IEEE Journal of Selected Topics in Signal Processing},
  vol.~9, no.~4, pp. 647--662, 2015.

\bibitem{rakhlin2013online}
A.~Rakhlin and K.~Sridharan, ``Online learning with predictable sequences,'' in
  \emph{Proceedings of the 26th Annual Conference on Learning Theory (COLT)},
  2013, pp. 1--27.

\bibitem{besbes2015non}
O.~Besbes, Y.~Gur, and A.~Zeevi, ``Non-stationary stochastic optimization,''
  \emph{Operations research}, vol.~63, no.~5, pp. 1227--1244, 2015.

\bibitem{jadbabaie2015online}
A.~Jadbabaie, A.~Rakhlin, S.~Shahrampour, and K.~Sridharan, ``Online
  optimization: Competing with dynamic comparators,'' in \emph{Artificial
  Intelligence and Statistics}, 2015, pp. 398--406.

\bibitem{yi2016tracking}
T.~Yang, Z.~Lijun, R.~Jin, and J.~Yi, ``Tracking slowly moving clairvoyant:
  Optimal dynamic regret of online learning with true and noisy gradient,'' in
  \emph{Proceedings of the 33rd International Conference on Machine Learning},
  2016, pp. 1--12.

\bibitem{shahrampour2017distributed}
S.~Shahrampour and A.~Jadbabaie, ``Distributed online optimization in dynamic
  environments using mirror descent,'' \emph{IEEE Transactions on Automatic
  Control}, vol.~63, no.~3, pp. 714--725, 2017.

\bibitem{lesage2018predictive}
A.~Lesage-Landry, I.~Shames, and J.~A. Taylor, ``Predictive online convex
  optimization,'' \emph{Automatica}, vol. 113, 2020.

\bibitem{li2018online}
Y.~Li, G.~Qu, and N.~Li, ``Online optimization with predictions and switching
  costs: Fast algorithms and the fundamental limit,'' \emph{arXiv preprint
  arXiv:1801.07780}, 2018.

\bibitem{zhao2018proximal}
Y.~Zhao, S.~Qiu, and J.~Liu, ``Proximal online gradient is optimum for dynamic
  regret,'' \emph{arXiv preprint arXiv:1810.03594}, 2018.

\bibitem{bernstein2019online}
A.~Bernstein, E.~Dall'Anese, and A.~Simonetto, ``Online primal-dual methods
  with measurement feedback for time-varying convex optimization,'' \emph{IEEE
  Transactions on Signal Processing}, vol.~67, no.~8, pp. 1978--1991, 2019.

\bibitem{zavala2010real}
V.~M. Zavala and M.~Anitescu, ``Real-time nonlinear optimization as a
  generalized equation,'' \emph{SIAM Journal on Control and Optimization},
  vol.~48, no.~8, pp. 5444--5467, 2010.

\bibitem{dontchev2013euler}
A.~L. Dontchev, M.~Krastanov, R.~T. Rockafellar, and V.~M. Veliov, ``An
  {E}uler--{N}ewton continuation method for tracking solution trajectories of
  parametric variational inequalities,'' \emph{SIAM Journal on Control and
  Optimization}, vol.~51, no.~3, pp. 1823--1840, 2013.

\bibitem{simonetto2016class}
A.~Simonetto, A.~Mokhtari, A.~Koppel, G.~Leus, and A.~Ribeiro, ``A class of
  prediction-correction methods for time-varying convex optimization,''
  \emph{IEEE Transactions on Signal Processing}, vol.~64, no.~17, pp.
  4576--4591, 2016.

\bibitem{tang2018feedback}
Y.~Tang, E.~Dall'Anese, A.~Bernstein, and S.~H. Low, ``A feedback-based
  regularized primal-dual gradient method for time-varying nonconvex
  optimization,'' in \emph{2018 IEEE Conference on Decision and Control (CDC)},
  2018, pp. 3244--3250.

\bibitem{tang2017real}
Y.~Tang, K.~Dvijotham, and S.~Low, ``Real-time optimal power flow,'' \emph{IEEE
  Transactions on Smart Grid}, vol.~8, no.~6, pp. 2963--2973, 2017.

\bibitem{kreyszig1978introductory}
E.~Kreyszig, \emph{Introductory functional analysis with applications}.\hskip
  1em plus 0.5em minus 0.4em\relax Wiley New York, 1978, vol.~1.

\bibitem{bertsekas1999nonlinear}
D.~P. Bertsekas, \emph{Nonlinear programming: 2nd Edition}.\hskip 1em plus
  0.5em minus 0.4em\relax Athena Scientific Belmont, 1999.

\bibitem{nocedal2006numerical}
J.~Nocedal and S.~Wright, \emph{Numerical optimization}.\hskip 1em plus 0.5em
  minus 0.4em\relax Springer Science \& Business Media, 2006.

\bibitem{boyd2004convex}
S.~Boyd and L.~Vandenberghe, \emph{Convex optimization}.\hskip 1em plus 0.5em
  minus 0.4em\relax Cambridge university press, 2004.

\bibitem{neely2017online}
M.~J. Neely and H.~Yu, ``Online convex optimization with time-varying
  constraints,'' \emph{arXiv preprint arXiv:1702.04783}, 2017.

\bibitem{paternain2017online}
S.~Paternain and A.~Ribeiro, ``Online learning of feasible strategies in
  unknown environments,'' \emph{IEEE Transactions on Automatic Control},
  vol.~62, no.~6, pp. 2807--2822, 2017.

\bibitem{lesage2018online}
A.~Lesage-Landry, H.~Wang, I.~Shames, P.~Mancarella, and J.~Taylor, ``Online
  convex optimization of multi-energy building-to-grid ancillary services,''
  \emph{IEEE Transactions on Control Systems Technology}, 2019.

\end{thebibliography}
\end{document}